\documentclass[a4paper,11pt]{amsart}
\usepackage[utf8x]{inputenc}
\usepackage{amsthm,amsfonts,amsmath,amssymb,enumerate}

\usepackage{cancel}

\usepackage{multirow}
\usepackage{color}

\renewcommand{\L}{\mathcal{L}}
\renewcommand{\S}{\mathcal{S}}
\newcommand{\N}{\mathcal{N}}

\newcommand{\K}{\mathbb{K}}
\newcommand{\GL}{\operatorname{GL}}
\renewcommand{\O}{\mathcal{O}}
\newcommand{\R}{\mathbb{R}}
\newcommand{\C}{\mathbb{C}}
\newcommand{\NN}{\mathbb{N}}
\newcommand{\g}{\mathfrak{g}}
\newcommand{\n}{\mathfrak{n}}
\newcommand{\h}{\mathfrak{h}}

\renewcommand{\sl}{\mathfrak{sl}}

\renewcommand{\mp}{\centerdot}
\newcommand{\ad}{\operatorname{ad}}
\renewcommand{\a}{\mathfrak{a}}

\newcommand{\Hom}{\operatorname{Hom}}

\newcommand{\tr}{\operatorname{tr}}
\renewcommand{\l}{\lambda}
\newcommand{\on}{\operatorname}

\newcommand {\Der} {\operatorname{Der}}
\renewcommand{\sc}{\circlearrowleft}
\newcommand{\vp}{\varphi}
\newcommand{\s}{\sigma}
\newcommand{\x}{\mathbf{x}}

\newtheorem{theorem}{Theorem}[section]
\newtheorem{lemma}[theorem]{Lemma}
\newtheorem{corollary}[theorem]{Corollary}

\newtheorem{proposition}[theorem]{Proposition}

\theoremstyle{definition}

\newtheorem{example}[theorem]{Example}

\theoremstyle{remark}
\newtheorem{remark}[theorem]{Remark}

\newtheorem{notation}[theorem]{Notation}
\newtheorem*{acknowledgements}{Acknowledgements}

\numberwithin{equation}{section}

\begin{document}

\title[]{Deformations and rigidity in varieties \\ of Lie algebras}
\author{Josefina Barrionuevo $\dag$ and Paulo Tirao $\dag$ $\ddag$
\smallskip\\
(\MakeLowercase{appendix by} Diego Sulca $\dag$)}
\address{$\dag$ CIEM-FaMAF, CONICET-Universidad Nacional de Córdoba \\
         Ciudad Universitaria, 5000 Córdoba, Argentina}
\address{$\ddag$ Guangdong Technion Israel Institute of Technology \\
         241 Daxue Road, Jinping District, Shantou, Guandong Province,
         China}
\date{July, 2022}

\begin{abstract}
We present a novel construction of linear deformations for Lie algebras and use it to prove 
the non-rigidity of several classes of Lie algebras in different varieties.
In particular, we address the problem of $k$-rigidity for $k$-step nilpotent Lie algebras and $k$-solvable Lie algebras.

We show that Lie algebras with an abelian factor are not rigid,
even for the case of a 1-dimensional abelian factor.
This holds in the more restricted case of $k$-rigidity.

We also prove that the $k$-step free nilpotent Lie algebras 
are not $(k+1)$-rigid, but however they are $k$-rigid. 
\end{abstract}

\maketitle



\section{Introduction}

Let $\K$ be an algebraically closed field of characteristic zero. 
The variety $\L_n$ of $n$-dimensional Lie algebras 
over $\K$ is the affine algebraic variety of all antisymmetric bilinear maps 
$\mu:\K^n\times\K^n\to\K^n$ which satisfy the Jacobi identity, 
called Lie brackets over $\K$.
The orbits of the natural action of $\GL(\K^n)$ by change of basis are
the isomorphism classes of Lie brackets. A Lie bracket $\mu$ is called \emph{rigid} if its orbit is Zariski open, 
or equivalently $\mu$ is not rigid if and only if in any Zariski neighborhood of it there is a non-isomorphic Lie bracket. 

Determining all $n$-dimensional rigid Lie algebras is an enormous and highly relevant problem that is out of reach today. 
There are a finite number of them and the closure of the orbit of a rigid bracket is an irreducible component of $\L_n$. 

Different problems concerning the variety of Lie algebras have been addressed quite extensively for a long time.
Determining their irreducible components and their rigid points
as understanding degenerations and deformations have been some of the goals for many authors.
The reader may look at the following shortlist, which is far from exhaustive in any sense, and the references therein: 
\cite{BS, C, GA, GH1, GT1, GT2, S, TV, V}.
The general picture becomes even more interesting if one looks to 
different subvarieties of $\L_n$ and address the same problems there.

We consider with special interest the subvariety $\N_{n}$ of $n$-dimensional nilpotent Lie algebras and the descending chain of subvarieties 
$\N_{n,k}$, for $k= n-1,\dots,1$, of $n$-dimensional nilpotent Lie algebras with nilpotency index less than or equal to $k$.
Notice that $\N_{n,n-1}=\N_n$ and that the complement of $\N_{n,n-2}$ 
inside $\N$ is the open subvariety of $n$-dimensional filiform Lie algebras
introduced and studied by M.\ Vergne \cite{V}.
We also consider the subvariety $\S_n$ of $n$-dimensional solvable Lie algebras and the corresponding chain $\S_{n,k}$ of $n$-dimensional solvable Lie algebras with solvability index
less than or equal to $k$.

A classical theorem by Nijenhuis and Richardson states that if 
the second Chevalley-Eilenberg adjoint cohomology group vanishes ($H^2(\mu,\mu)=0$), 
then $\mu$ is rigid in $\L_n$ and the first example shows that 
the reciprocal does not hold was provided in \cite{R}.
This result can be adapted to extend it to other varieties of Lie algebras. 
A general strategy is discussed by Remm in \cite[Sections 2.2 and 2.3]{Remm}.
We include a self-contained proof of this fact in the Appendix
and make explicit the corresponding statements for the varieties
$\N_{n,k}$ and $\S_{n,k}$.
For $\mu\in\N_{n,k}$, the vanishing of the space 
\begin{align*}
H^2_{k\textrm{-nil}}(\mu,\mu)&=\frac{Z_{N_{n,k}}^2(\mu,\mu)}{B^2(\mu,\mu)}\\
&=\frac{Ker(\delta)\bigcap Ker(\eta_k)}{Im(\delta^1)},
\end{align*}
with $\delta:\Lambda^2({\K^n}^*)\rightarrow\Lambda^3({\K^n}^*)$ given by: 
$$
\delta\omega(x,y,z):=\sc \mu(\omega(x,y),z)+\sc \omega(\mu(x,y),z),
$$ 
$\delta^1:\Lambda^1({\K^n}^*)\rightarrow\Lambda^2({\K^n}^*)$ given by:
$$\delta^1(f)(x,y)=\mu(f(x),y)+\mu(x,f(y))-f(\mu(x,y)),$$
and $\eta_k$ given by: 
$$\eta_k(\omega)=
\sum_{j=0}^{k-1}\mu^{k-1-j}\circ\omega\circ\mu^{j},
$$
implies the rigidity of $\mu$ in $\N_{n,k}$.

This description of $H^2_{k\textrm{-nil}}(\mu,\mu)$ was given
in \cite{BCC} in a slightly different form and in a differential geometry context.
In \cite{GR1} the instances for $k=2$ and $k=3$ were also discussed. 

Semisimple Lie algebras are rigid by Whitehead's Lemma,
and a semisimple Lie algebra $\g$ plus a 1-dimensional abelian 
factor $\a$ is also rigid, since its second cohomology group vanishes. 
This fact follows from the Hochschild-Serre spectral sequence
associated with the ideal $\g$ of $\g\oplus\a$.
Also, Borel subalgebras of semisimple Lie algebras have null second cohomology group \cite{LL} 
and hence are rigid. 

A natural and very interesting open question is: 
\begin{itemize}
	\item Are there nilpotent rigid Lie algebras in $\L_n$?
\end{itemize} 
This question, known since 1970 as Vergne's conjecture, 
has not been answered yet.
We believe that the answer is \emph{no}.
This paper, in particular, adds support to our beliefs.

The following stronger versions of this question are also challenging:
\begin{itemize}
	\item Are there $k$-step nilpotent rigid Lie algebras in $\N_{n,k}$?
	
	\item Are there $k$-step nilpotent rigid Lie algebras in $\N_{n,k+1}$?
\end{itemize}

The 3-dimensional Heisenberg Lie algebra is rigid in $\N_{3,2}=\N_{3}$.
Besides this small dimension example we prove that the free $k$-step nilpotent 
Lie algebra on $m$ generators, $L_{(k)}(m)$, is rigid in $\N_{n,k}$ where $n=\dim L_{(k)}(m)$. 
The result follows by proving that their second nil-cohomology vanishes.
In an analogous way one can show the $(2n+1)$-dimensional Heisenberg Lie 
algebra is rigid in $\N_{2n+1,2}$, something that was proved in \cite{GR1} 
and \cite {A} by different means.

Regarding the second question, up to our knowledge based on existing 
examples, its answer is (in general) \emph{no}.
In this paper, we provide further classes of examples by constructing
non-trivial linear deformations.
In particular we show that the free $k$-step nilpotent Lie algebra
$L_{(k)}(m)$ is not rigid $\N_{n,k+1}$ where $n=\dim L_{(k)}(m)$
and similarly the $(2n+1)$-dimensional Heisenberg Lie algebra is not rigid in 
$\N_{2n+1,3}$.

The construction of non-trivial deformations is done by using a novel
construction of linear deformations that we present in Section 3.
With this tool, we tackle the rigidity problem for Lie algebras with
an abelian factor, showing that in general, they are non-rigid.
More precisely, all $k$-solvable Lie algebras plus an abelian
factor are non-rigid in the corresponding variety $\S_{n,k}$
and all $k$-nilpotent Lie algebras plus an abelian factor are non-rigid
in the corresponding variety $\N_{n,k}$, with the only exception of
the 3-dimensional Heisenberg Lie algebra plus a 1-dimensional abelian factor.

\section{Some preliminaries}

In this paper, $n$ is a fixed natural number and $\K$ an algebraically closed field of characteristic zero.

We consider $n$-dimensional Lie $\K$-algebras.
Through the whole paper, we will refer to a Lie algebra $\g$
or to its Lie bracket $\mu$ indistinctly, according to which notation fits
the exposition better.
We shall mainly use $\mu$ and use $\g$ when the underlying vector space is relevant, for instance, to refer to a subalgebra. We may also write $(\g,\mu)$.

\subsection{Multilinear maps}

Let $V$ be a vector space and let $C^i(V)=\Hom(V^{\otimes i},V)$ be the space of $i$-multilinear maps from $V\times\cdots\times V$ to $V$.
Given $\varphi\in C^i(V)$ and $\psi\in C^j(V)$ let $\varphi\circ\psi \in C^{i+j-1}(V)$ be the multilinear map
defined by
\[ \varphi\circ\psi(x_1,\dots,x_{i+j-1})=\varphi(\psi(x_1,\dots,x_j),x_{j+1},\dots,x_{i+j-1}). \]
Also we define inductively
\[ \varphi^k=\varphi\circ \varphi^{k-1}. \]
Notice that if $f:V\rightarrow V$ is linear and $\varphi\in C^i(V)$, then $f\circ\varphi\in C^i(V)$.

In particular, if $\mu\in C^2(V)$, then $\mu\circ\mu$ is the trilinear map given by
\[ \mu\circ\mu(x,y,z)=\mu(\mu(x,y),z). \]
For a trilinear map $\varphi$ we write
\[ \sc \varphi(x,y,z)=\varphi(x,y,z)+\varphi(y,z,x)+\varphi(z,x,y). \]
Hence, for a bilinear map $\mu$, we have that
\[ \sc \mu\circ\mu (x,y,z)=\mu(\mu(x,y),z)+\mu(\mu(y,z),x)+\mu(\mu(z,x),y). \]
So the Jacobi identity for $\mu$ is
\[ \sc \mu\circ\mu = 0. \]

Given $f,g,h\in V^*$, the dual linear space of $V$, 
$f\cdot g:V\times V\rightarrow \K$ and 
$f\cdot g\cdot h:V\times V\times V\rightarrow \K$ are the
bilinear and trilinear maps defined by
\[ f\cdot g\ (x,y)=f(x)g(y) \quad\text{and}\quad f\cdot g\cdot h\ (x,y,z)=f(x)g(y)h(z). \]

A direct computation yields the following result that we shall use 
in the next section.
\begin{lemma}\label{lemma:sc}
Given $f,g,h\in V^*$, it holds that $\sc(f\cdot g-g\cdot f)\cdot f=0$,
$\sc(f\cdot g-g\cdot f)\cdot g=0$ and 
$\sc(f\cdot g-g\cdot f)\cdot h=\sc(f\cdot h-h\cdot f)\cdot g$.
\end{lemma}
\color{black}

\subsection{Varieties of Lie algebras}

Let $V=\K^n$.
The variety $\L_n$ of $n$-dimensional Lie $\K$-algebras is the
affine algebraic variety of all the antisymmetric maps $\mu\in C^2(V)$ satisfying $\sc \mu\circ\mu=0$.

The subvariety $\N_{n}$ of $n$-dimensional nilpotent Lie $\K$-algebras, 
is composite by those Lie brackets $\mu$ such that $\mu^j=0$, 
for some $j\ge 1$.
A Lie algebra $\mu$ is said to be $k$-step nilpotent, for $k\ge 2$, 
if $\mu^k=0$ and $\mu^{k-1}\ne 0$.
We consider the abelian Lie algebra $\mu=0$ as 1-step nilpotent.
The subvariety $\N_{n,k}$ of $n$-dimensional nilpotent Lie algebras at most $k$-step
nilpotent is then composite by all Lie brackets $\mu$ such that 
$\mu^k=0$.
Notice that $\N_{n,k}\subset\N_{n,k+1}$ and that $\N_{n,n-1}=\N$.

The subvarieties of solvable and $k$-step solvable Lie algebras
are defined analogously by considering $\mu^{(k)}$ instead of $\mu^k$,
where $\mu^{(1)}=\mu$ and for $k\ge 2$
\[ \mu^{(k)}=\mu\big(\mu^{(k-1)},\mu^{(k-1)}\big). \]

\paragraph{\emph{Orbits and rigidity}}
The orbit $\O(\mu)$ in $\L_n$ of a Lie bracket $\mu$ under the action of
$\operatorname{GL}_n(\K)$ given by \emph{change of basis} is the isomorphism class of $\mu$.
Clearly, if $\mu$ is in any of the subvarieties described above,
its orbit $\O(\mu)$ is contained in it. 
A Lie algebra $\mu$ in a subvariety $\L'$ of these is said to be rigid in $\L' $
if its orbit is open in $\L'$.

\smallskip

\paragraph{\emph{$k$-rigidity for nilpotent Lie brackets}}
Given a nilpotent Lie bracket $\mu$ of $\K^n$,
we say that it is \emph{$k$-rigid}, if it is rigid
in $\N_{n,k}$.

Given a $k$-step nilpotent Lie bracket $\mu$ of $\K^n$ we may ask for the
smallest $j$, $j\ge k$, such that $\mu$ is not rigid in $\N_{n,j}$,
if such a $j$ exists.
According to the available evidence we have, in general, the 
smallest $j$ is just $k+1$. 
That is, there is no known $k$-step nilpotent Lie algebra
which is $(k+1)$-rigid.
This paper adds more evidence to the negative answer of the question:
\begin{itemize}
	\item Are there $k$-step nilpotent Lie brackets which are 
	$(k+1)$-rigid?
\end{itemize}
To prove $k$-rigidity, we will use the Corollary \ref{coro:k-rig} of the appendix. To prove non-$k$-rigidity, we will construct non-trivial linear deformations.

\subsection{Linear deformations}\label{subsec:lin-def}

For a general presentation of the theory of deformations, including Lie algebras, the reader may refer to \cite{MM}.

Given a Lie bracket $\mu$ of $\K^n$, we shall consider 
\emph{linear deformations} of $\mu$, that is deformations
of the form $\mu_t = \mu + t\varphi$, where t is a parameter in $\K$.
It is straightforward to verify that $\mu_t$ is a Lie algebra for all $t$ if and only if $\varphi$ is a Lie algebra and a 2-cocycle for 
$\mu$, that is $\sc\mu\circ\varphi+\varphi\circ\mu=0 $.

\

\paragraph{\emph{Grunewald-O'Halloran construction}}

Given a Lie algebra $\g$ with bracket $\mu$, an ideal $\h$ of codimension 1 and a derivation $D$ of $\h$, Grunewald and O’Halloran \cite{GH2} considered the linear deformation of $\mu$
\[ \mu_t = \mu + t\varphi_D, \]
where $\varphi_D$ is the (Dixmier) 2-cocycle defined by
\[ \varphi_D(x,h) = D(h)=-\varphi_D(h,x), \quad \varphi_D(h,h') = 0, \]
for $h,h'\in\h$ and $x$ a fixed element outside of $\h$.
Notice that $\h$ remains an ideal of $\mu_t$, for all $t$.

\section{A novel construction of linear deformations}

In that follows, we construct linear deformations of a given Lie algebra $\g$ with Lie bracket $\mu$ starting from a subalgebra $\h$ of $\g$ of codimension 2.

Fix $a_1,a_2\in\g$ such that $\langle a_{1},a_{2}\rangle$ is a complementary subspace to $\h$, i.e.\
\[ \g=\langle a_{1},a_{2}\rangle\oplus\h.\]
For a basis $\{h_1,\dots,h_{n-2}\}$ of $\h$, 
\[B=\{a_1,a_2,h_1,\dots,h_{n-2}\}\]
 is a basis of $\g$.
Let $B^*=\{a_1^*,a_2^*,h_1^*,\dots,h_{n-2}^*\}$ be the dual basis
of $B$.
Hence, given $x\in\g$, there is a unique $x_\h\in\h$ such that
\[ x=a_1^*(x) a_{1}+a_2^*(x) a_{2}+x_\h. \]
We denote the linear map $x\mapsto x_h$ by $\pi_\h$.

By $\ad_x$ we denote the adjoint of $x\in\g$, $\ad_x:\g\to\g$.
For $h\in\h$, we denote by $\ad_h^{\h}$ the adjoint of $h$ restricted to $\h$, $\ad_h^{\h}:\h\to\h$.

In addition, for $y\in\g$, we shall consider the antisymmetric bilinear map $\varphi=a_{1}^*\wedge a_{2}^*\otimes y$\color{black}. 
Recall that, for all $u,v\in\g$,
\begin{eqnarray*}
  (a_{1}^*\wedge a_{2}^*\otimes y)(u,v) &=& 
    (a_1^*\cdot a_2^*-a_2^*\cdot a_1^*)(u,v) y \\
    &=& (a_1^*(u)a_2^*(v)-a_2^*(u)a_1^*(v))y.
\end{eqnarray*}
In particular $\vp(\h,\g)=0$. Finally notice that $\varphi$ is a Lie bracket isomorphic to a $3$-dimensional Heisenberg Lie algebra 
plus an $(n-3)$-dimensional abelian Lie algebra so that $\varphi$ 
is a Lie bracket.

\begin{theorem}\label{thm:main-construction}
Let  $(\g,\mu)$ be a Lie algebra and $\h \subseteq \g$ a subalgebra of codimension 2. 
Fix $a_1,a_2\in\g$ such that $\g= \langle a_{1},a_{2}\rangle\oplus\h$ and $a_1^*\circ\ad_{a_1} + a_2^*\circ\ad_{a_2} = 0$ in $\h$. Then for any $y\in Z_\g(\h)$,
\[ \mu_t=\mu + t (a_{1}^*\wedge a_{2}^*\otimes y) \]
is a linear deformation of $\mu$. 
\end{theorem}

\begin{proof}
Let us denote $\mu=[\ ,\ ]$ and $a_{1}^*\wedge a_{2}^*\otimes y=\vp$.
Since $\vp$ is already a Lie bracket, it remains to show that $\vp$
is a 2-cocycle for $\mu$, that is
\[ \sc \left([\ ,\ ] \circ \varphi + \varphi \circ [\ ,\ ]\right) =0. \]
We show that moreover $\sc [\ ,\ ] \circ \varphi=0$ and 
$\sc \varphi \circ [\ ,\ ]=0$.

On the one hand, let $u,v,w\in\g$. $\left([\ ,\ ]\circ\vp\right)(u,v,w)=[\vp(u,v),w]$ and 
\begin{eqnarray*}
 [\vp(u,v),w] &=& \left[\ (a_1^*a_2^*-a_2^*a_1^*)(u,v)\  y\ ,\ a_1^*(w)a_1+a_2^*(w)a_2+w_\h\ \right] \\
              &=& ((a_1^*a_2^*-a_2^*a_1^*)a_1^*)(u,v,w)\ [y,a_1] \\
              && \qquad +((a_1^*a_2^*-a_2^*a_1^*)a_2^*)(u,v,w)\ [y,a_2] \\
              && \qquad + (a_1^*a_2^*-a_2^*a_1^*)(u,v)\ [y,w_\h] \\
              &=& 0.
\end{eqnarray*}
The first two terms are equal to $0$ by Lemma \ref{lemma:sc} and
the third one is $0$ because $y\in Z_\g(\h)$.

On the other hand, let $u,v,w\in\g$, $(\vp\circ[,])(u,v,w)=\vp([u,v],w)$.
Writing 
\begin{gather*}
u=a_1^*(u)a_1+a_2^*(u)a_2+u_\h \\
v=a_1^*(v)a_1+a_2^*(v)a_2+v_\h,
\end{gather*}
it follows that
\begin{eqnarray*}
[u,v] &=& \left(a_1^*(u)a_2^*(v)-a_2^*(u)a_1^*(v)\right)[a_{1},a_{2}] +[u_\h,v_\h]+ 
a_1^*(u)[a_{1},v_\h] \\
      && \qquad + a_2^*(u)[a_{2},v_\h] + a_1^*(v)[u_\h,a_{1}] + a_2^*(v)[u_\h,a_{2}],
\end{eqnarray*}
and since $[u_\h,v_\h]\in\h$ we have that
\begin{eqnarray*}
\vp([u,v],w) &=& (a_1^*(u)a_2^*(v)-a_2^*(u)a_1^*(v))\vp([a_{1},a_{2}],w) + a_1^*(u)\vp([a_1,v_\h],w) \\
             &+& a_2^*(u)\vp([a_{2},v_\h],w) - a_1^*(v)\vp([a_{1},u_\h],w) - a_2^*(v)\vp([a_{2},u_\h], w).
\end{eqnarray*}
The first term is equal to
\begin{eqnarray*}
 S_1(u,v,w) &=& ((a_1^*\cdot a_2^*-a_2^*\cdot a_1^*)\cdot a_2^*)(u,v,w)\ a_1^*([a_{1},a_{2}])\ y \\
  && \qquad - ((a_1^*\cdot a_2^*-a_2^*\cdot a_1^*)\cdot a_1^*)(u,v,w)\ a_2^*([a_{1},a_{2}])\ y;
\end{eqnarray*}
the sum of the second and forth terms is equal to 
\begin{eqnarray*}
S_{24}(u,v,w) &=& (a_1^*\cdot(a_1^*\circ\ad_{a_1}\circ\pi_\h)-(a_1^*\circ\ad_{a_1}\circ\pi_\h)\cdot a_1^*)\cdot a_2^*\ (u,v,w)\ y+\\
&& \qquad ((a_2^*\circ\ad_{a_1}\circ\pi_\h)\cdot a_1^*-a_1^*
               \cdot (a_2^*\circ\ad_{a_1}\circ\pi_\h))\cdot a_1^*\ (u,v,w)\ y;
\end{eqnarray*}
and the sum of the third and fifth terms is equal to
\begin{eqnarray*}
S_{35}(u,v,w) &=& (a_2^*\cdot (a_1^*\circ\ad_{a_2}\circ\pi_\h)
     -(a_1^*\circ\ad_{a_2}\circ\pi_\h)\cdot a_2^*)\cdot a_2^*\ (u,v,w)\ y+\\
&& \qquad ((a_2^*\circ\ad_{a_2}\circ\pi_\h)\cdot a_2^*
      -a_2^*\cdot (a_2^*\circ\ad_{a_2}\circ\pi_\h))\cdot a_1^*\ (u,v,w)\ y.	
\end{eqnarray*}
Now we have that, by Lemma \ref{lemma:sc},
\[ \sc S_1(u,v,w)=0. \]
Also by the same lemma, 
\begin{eqnarray*}
  \sc S_{24}(u,v,w) &=& \sc (a_1^*\cdot(a_1^*\circ\ad_{a_1}\circ\pi_\h)
      -(a_1^*\circ\ad_{a_1}\circ\pi_\h)\cdot a_1^*)\cdot a_2^*\ (u,v,w)\ y \\
  \sc S_{35}(u,v.w) &=& \sc ((a_2^*\circ\ad_{a_2}\circ\pi_\h)\cdot a_2^*
      -a_2^*\cdot (a_2^*\circ\ad_{a_2}\circ\pi_\h))\cdot a_1^*\ (u,v,w)\ y.
\end{eqnarray*}
Finally, by hypothesis  
$a_1^*\circ\ad_{a_1}\circ\pi_\h=-a_2^*\circ\ad_{a_2}\circ\pi_\h$. 
This implies, by using Lemma \ref{lemma:sc}, that
\[ \sc (S_{24}+S_{35})=0 \]
and therefore

\[ \sc \vp \circ [\ ,\ ] = 0 \]

as we wanted to prove.

\end{proof}

\begin{corollary}\label{coro:main-construction}
Let  $(\g,\mu)$ be a nilpotent Lie algebra and 
 $\h \subseteq \g$ a subalgebra  of codimension 2. Fix $a_1,a_2\in\g$ be such that
$\g=\langle a_{1},a_{2}\rangle\oplus\h$. \color{black}
Then for any $y\in Z_\g(\h)$ 
$$\mu_t=\mu + t (a_{1}^*\wedge a_{2}^*\otimes y)$$
is a linear deformation of $\mu$.
\end{corollary}
 
\begin{proof}
Both $\g$ and $\h$ are nilpotent Lie algebras, hence 
$\tr(\ad_x)=0$, for all $x\in \g$ and $\tr(\ad_h^{\h})=0$, 
for all $h\in \h$.
Since 
\[ \tr(\ad_h)=a_{1}^*([h,a_1])+a_{2}^*([h,a_2])+\tr(\ad_h^{\h}), \] 
it follows that

\[
0=a_1^*\circ\ad_{a_1} (h) + a_2^*\circ\ad_{a_2}(h).
\]
Thus the hypotheses of Theorem \ref{thm:main-construction} are satisfied. 
\end{proof}

\begin{remark}
If $\g$ is  nilpotent, this construction is a particular case of the Grunewald-O'Halloran construction \cite{GH2}.
This follows from two easy arguments. 
First, all subalgebras of codimension $2$ can be extended to an ideal of codimension $1$.
[Given $\langle a_1,a_2\rangle$ a direct linear complement of $\h$ in $\g$, then
either $a_1$ or $a_2$ are not in $[\g,\g]$. In fact if both are in $[\g,\g]$, since $\g$ is nilpotent and $\h$ is a subalgebra, 
then $a_1=[a_2,h_1]$ and $a_2=[a_1,h_2]$ for some $h_1,h_2\in\h$. But this is not possible for $\g$ nilpotent.]
Then we can assume that $\g=\langle a_1\rangle\oplus I$, with $\h\subseteq I$. 
Finally, the linear function of $I$ that sends $a_2$ to $y$ and the rest of elements of a basis to $0$, is a derivation of $I$.
\end{remark}

In general, our construction is different from that in \cite{GH2}, as the following two examples show.

\begin{example}
Let $(\g=\langle a_1,a_2,y\rangle,\mu)$ with Lie bracket 
\begin{gather*}
\mu(a_1,y)=2a_1\\
\mu(a_2,y)=2a_2\\
\mu(a_1,a_2)=0
\end{gather*}

If we take $\h=\langle y\rangle$, it satisfies the hypotheses of the Theorem \ref{thm:main-construction}. It follows that $\mu_t$ is isomorphic to $\sl_2(\K)$ for all $t\ne 0$.
Since $\sl_2(\K)$ has no ideals, the deformation $\mu_t$ is not Grunewald-O'Halloran's type.
\end{example}

\begin{example} 
Let $(\h,\nu)$ be a non-perfect Lie algebra with non-trivial center 
and let $f:\h\rightarrow \K$ be  a non-zero linear map such that $f(\nu(\h,\h))=0$.
Define the Lie algebra $(\g,\mu)$ by taking 
$\g=\langle a_1,a_2 \rangle \oplus \h$ with $\mu$ defined by:
\begin{gather*}
\mu_{|_{\h\times\h}}=\nu, \\
\mu(a_1,h)=f(h)a_2, \\
\mu(a_2,h)=f(h)a_1, \\
\mu(a_1,a_2)=0,
\end{gather*}
for all $h\in\h$. 
By taking $y\neq 0\in Z(\h)$, the hypotheses of 
Theorem \ref{thm:main-construction} are satisfied. 
The corresponding linear deformation $\mu_t$ is not of 
Grunewald-O'Halloran type.

Indeed, if $\mu_t=\mu + t (a_{1}^*\wedge a_{2}^*\otimes y)$ were of that type, it would exist an ideal 
$I\triangleleft\g$ of codimension one, $x\in\g$ and $D\in \Der(I)$ such that $\g=\langle x\rangle\oplus I$ and  
\begin{gather}\label{eqn:mulocal}
\mu_t(i_1,i_2)=\mu(i_1,i_2)\\
\mu_t(x,i) =\mu(x,i)+tD(i)
\end{gather}
for all $i, i_1,i_2 \in I$.
Since $I\triangleleft\g$ has codimension one, $\mu(\g,\g)\subseteq I$  and then $a_1,a_2\in I$.
Now from \eqref{eqn:mulocal} we get that 
\[ \mu_t(a_1,a_2)=\mu(a_1,a_2). \]  
However by the definition of $\mu_t$, we have that
\[ \mu_t(a_1,a_2)
=(\mu + t (a_{1}^*\wedge a_{2}^*\otimes y))(a_1,a_2)=\mu(a_1,a_2)+ty,
\]
which is not possible for $t\ne 0$. 
Therefore $\mu_t$ is not of Grunewald-O'Halloran type.
\end{example}

\subsection{2-step nilpotent graph Lie algebras}

As an example of the previous construction, we deform 2-step nilpotent
graph Lie algebras (see \cite{AAA,BT} for 2-step graph Lie algebras in the degenerations and deformations
framework).

Let $G$ be the graph with vertices $V=\{v_1,\dots,v_m\}$ and
edges $A=\{a_{ij}:(i,j)\in I\}$, $I\subseteq \{1,\dots, n\}\times  \{1,\dots, n\}$.
The graph Lie algebra associated to $G$, $\g_G$, is the $\K$-vector space generated by $V\cup A$, 
where the non-zero brackets of basis elements are

\[
\mu(v_i,v_j)=a_{ij}, \quad\text{if } (i,j)\in I.
\]

Notice that if $A=\emptyset$, then $\g_G=\a_m$, the $m$-dimensional
abelian Lie algebra. 
But in general, if $A\neq \emptyset$, it is a $2$-step nilpotent Lie algebra, i.e $\g_G\in\N_{n,2}$, $n=|V|+|A|$. 

\begin{example}
	The Lie algebra associated with the graph with two vertices and one edge
	is the 3-dimensional Heisenberg Lie algebra $\h_1$.
\end{example}

In general, 2-step nilpotent graph Lie algebras are not 3-rigid, 
as stated precisely in the following theorem.
\begin{theorem}
	Let $\g$ be a $n$-dimensional graph Lie algebra non isomorphic to $\h_1$, $\a_1$ or 
	$\a_2$, then $\g$ is non-rigid in $\N_{n,3}$.
\end{theorem}

\begin{proof}
	Let $\mu$ be the bracket of $\g$.
	In all cases, we construct a non-trivial 3-step nilpotent deformation of $\mu$.
	
	If $A=\emptyset$, $\g\simeq\a^n$ with $n\ge 3$ and $\mu=0$, 
	then [Section 6, item (3)] provides a non-trivial $2$-step deformation of $\mu$.
	
	If $|A|=1$, since $\g\ncong\h_1$, we have that $m>2$. 
	We may assume, by relabeling the vertices if necessary,
	that $A=\{a_{12}\}$.  
	The hypotheses of Corollary \ref{coro:main-construction} are fulfilled for $a_1=v_1$, $a_2=a_{12}$, $\h=\langle V-\{v_1\}\rangle$ and $y=v_3$. 
	Then $\mu_t=\mu + t(v_1\wedge a_{12} \otimes v_3)$ is a linear deformation of $\mu$, which is $3$-step nilpotent for all $t\ne 0$.
	
	If $|A|>1$, we can assume that $a_{12}\in A$ and that there is another edge in $a\in A$.
	The hypotheses of Corollary \ref{coro:main-construction} are fulfilled for $a_1=v_1$, $a_2=a_{12}$, $\h=\langle (V-\{v_1\})\cup (A-\{a_{12}\})\rangle$ and $y=a$. 
	Then $\mu_t=\mu + t(v_1\wedge a_{12} \otimes a)$ is a linear deformation of $\mu$, which is $3$-step nilpotent for all  
	$t\ne 0$.
\end{proof}

\section{Free nilpotent Lie algebras}

Let $L_{(k)}(m)$ be the free $k$-step nilpotent Lie algebra on $m$ generators\color{black}, where $m\ge 2$, and be $n$ its dimension.
In this section, we explore the rigidity of $L_{(k)}(m)$ in the varieties 
$\N_{n,k}$ and $\N_{n,k+1}$, showing that it is rigid in the first one,
but not in the second one.
There is a single exception: $L_{(2)}(2)$, this Lie algebra is isomorphic to the $3$-dimensional Heisenberg Lie algebra, which is rigid in $\N_{3,2}=\N_{3,3}=\N_3$.

\medskip

Let us recall briefly the construction of $L_{(k)}(m)$ and some well-known facts to fix notation.
Given a set of generators $X$, one constructs one after the other, 
the free magma $M(X)$, the free algebra $A(X)$, the free Lie algebra
$L(X)$ and finally the $k$-step free nilpotent Lie algebra on $X$, $L_{(k)}(X)$.

\textbf{The free magma on $X$} is the set $M(X)$ with an operation `$\centerdot$'.
\[ M(X)=\bigcup_{i\in\NN} M_i(X), \]
where $M_1(X)=X$ and for $i\geq 2$,  
$M_{i}(X)=\{p\centerdot q \ | \ p\in M_s(X), q\in M_t(X),s+t=i\}$.
The elements of $M_i(X)$ are said to be of length $i$.

\textbf{The free algebra on $X$} is the algebra $A(X)$ built on the linear space generated by the set $M(X)$
with the bilinear product induced by `$\centerdot$'.
Clearly, it is naturally graded
\[ A(X)=\bigoplus_{i=1}^\infty A_i(X), \]
where $A_i(X)=\langle M_i(X) \rangle$. 
Notice that $A_i(X)\centerdot A_j(X)\subseteq A_{i+j}(X)$.
The elements of $A_i(X)$ are said to be of degree $i$.

\textbf{The free Lie algebra on $X$} is the Lie algebra $(L(X),\lambda)$, constructed as the quotient algebra
\[ L(X)=\frac{A(X)}{I}, \]
where $I$ is the ideal generated by the set $$\{ a\centerdot b+b\centerdot a \ |\ a,b\in A(X) \}
\cup \{\sc\centerdot^2(a,b,c)\ |\ a,b,c\in A(X)\}.$$
The quotient projection is $\pi:A(X)\rightarrow L(X)$. 
Since the ideal $I$ is homogeneous, $L(X)$ is naturally graded
\[ L(X)=\bigoplus_{i=1}^\infty L_i(X), \]
where $L_i(X)=\pi(A_i(X))$.

\textbf{The $k$-step free nilpotent Lie algebra on $X$}  is the Lie algebra $(L_{(k)}(X),\mu)$, constructed as the quotient algebra
\[ L_{(k)}(X)=\frac{L(X)}{L^{k+1}(X)}, \]
where $L^{k+1}(X)$ is the $(k+1)$-th term of the central descending series of $L(X)$,
\[ L^{k+1}(X)=\bigoplus_{i=k+1}^\infty L_i(X). \]
The quotient projection is
$\overline{\phantom{I}}:L(X)\rightarrow L_{(k)}(X)$. 
Since the ideal $L^{k+1}(X)$ is homogeneous, 
$L_{(k)}(X)$ is naturally graded
\[ L_{(k)}(X)=\bigoplus_{i=1}^k L_i(X), \]
where, for $i\le k$, we identify $L_i(X)$ with its image in the quotient.
Finally, we denote by $\pi_k=\overline{\phantom{I}}\circ \pi$ 
the composition
\[ \pi_k:A(X)\longrightarrow L(X) \longrightarrow L_{(k)}(X). \]
Sometimes for convenience we will write $[\ ,\ ]$ instead of $\mu$.

\textbf{Hall bases.}\label{subsec:hall}
Let $X=\{x_1,\cdots,x_m\}$.
Since $L_i(X)=\pi(A_i(X))$ and $A_i(X)=\langle M_i(X)\rangle$, 
it follows that
\[ L_i(X)=\langle \pi(M_i(X))\rangle. \]
Linear basis for each $L_i(X)$, and hence for the free Lie algebra 
$L(X)$ and for the $k$-step free nilpotent Lie algebras, 
can be chosen from $\pi(M_i(X))$.
Very well-known bases of this kind are the \emph {Hall bases}.

\begin{enumerate}[(1)]
 
\item 
Starting from an ordered basis $B_1$ of $L_1(X)$ one constructs
recursively ordered basis $B_i$ of $L_i(X)$ as follows:
\begin{enumerate}[(i)]
 \item Let $B_1$ be the set of generators $X$ ordered by
   \[ x_1<x_2<\dots<x_m. \]
   
 \item Given ordered bases $B_1,\dots,B_k$ of $L_1(X),\dots,L_k(X)$
  respectively, $u\in B_i$ and $v\in B_j$, with $i\ne j$, are ordered by
   \[ u<v, \text{\qquad if $i<j$}. \]
   
 \item Now $B_{k+1}$ is formed by all brackets $\lambda(u,v)$, with
  $u\in B_i$, $v\in B_j$ with $i+j=k+1$, subject to the 
  following restrictions:
   \[ u>v \text{\qquad and \qquad if\quad} u=\lambda(w,z),\quad  v\ge z. \]
  The elements of $B_{k+1}$ are ordered lexicographically; that is
  $\lambda(u,v)>\lambda(w,z)$ if $u>w$, or $u=w$ and $v>z$.
   
 \item Given an order for the generators, the basis so constructed 
  is uniquely determined.  
\end{enumerate}

\item Observe that each element $v$ in the basis $B_i$ has a multidegree 
$D_v=(d_1,\dots,d_n)$, where $d_j$ is the number of occurrences of 
$x_j$ in $v$.
Clearly $d_1+\dots+d_n=i$ and moreover the bracket is multigraded,
that is $D_{[v,w]}=D_v+D_w$.

From the construction and the last observation, three facts,
that we recall for later use, follow\color{black}:
\begin{enumerate}[(i)]
 \item The element $c_k=\lambda^{k-1}(x_2,x_1,\dots,x_1)$, of multidegree 
 $D=(k-1,1,0,\dots,0)$ is in $B_k$, for every $k\ge 2$.
 
 \item  Given $v\in B_i$ and $w\in B_j$, $\lambda(v,w)\in L_{i+j}(X)$ is a linear combination of the those elements of $B_{i+j}$ with 
  multidegree equal to $D_v+D_w$. The coefficient of $c_k$, with $k=i+j$, in this linear combination is almost always zero. 
  The only exception is the case $v=c_{k-1}$ and $w=x_1$,
  in which case $\lambda(v,w)=c_k$.
 
 \item In general $\dim L_i(X)\ge 2$, with the only exception of
  $L_2(X)$ if $m=2$, in which case $\dim L_2(X)=1$.
\end{enumerate}

\end{enumerate}

\subsection{Main results}

Given $X=\{x_1,...,x_m\}$ let us denote $M(X)$, $A(X)$, $L(X)$ and $L_{(k)}(X)$ by $M(m)$, $A(m)$, $L(m)$ and $L_{(k)}(m)$ respectively. 

\begin{theorem}
The $n$-dimensional free $k$-step nilpotent Lie algebra on $m$ generators $L_{(k)}(m)$, 
is non-rigid in $\N_{n,k+1}$, if it is not isomorphic to
$L_{(2)}(2)$.
\end{theorem}
\begin{proof}
We construct a $(k+1)$-step linear deformation of $L_{(k)}(m)$ 
using Corollary \ref{coro:main-construction}.

Let $B$ be the Hall basis (given in the previous section) associated to the set of generators 
$\{x_1,...,x_m\}$ ordered by $x_1<\dots<x_m$. 
Let
\[ a_1=x_1 \qquad\text{and}\qquad 
   a_2=c_k=[[\dots[[x_2,x_1],x_1],\dots],x_1],
\] 
and $\h=\langle B-\{a_1,a_2\}\rangle$. The subspace $\h$ is in fact a subalgebra, as it follows from the item (2)(ii) of the previous facts.
Also $B_k$ has at least 2 elements, as it follows from the item (2)(iii) and
the fact that $L_{(k)}(m)\not\simeq L_{(2)}(2)$.
Then there exists $y\in B_k$ linearly independent with $a_2\in B_k$. Since $B_k\subseteq Z(L_{(k)}(m))$, 
then $y\in Z_{L_{(k)}(m)}(\h)$\color{black}. 
Hence, by Corollary \ref{coro:main-construction}, we can consider
the linear deformation of $L_{(k)}(m)$ given by 
\[ [\ ,\ ]_t=[\ ,\ ] + t(a_1\wedge a_2\otimes y). \]
If $t\ne 0$, $[a_1,a_2]_t=ty\ne 0$, so that 
$L_t^{k+1}=\langle y\rangle \ne 0$
and therefore $L_t$ is a $(k+1)$-step nilpotent Lie algebra for all $t\ne 0$. \end{proof}

The following fact is needed in our proof of the next theorem.
\begin{remark}\label{rmk:chains}
Given $x_{i_1},\dots,x_{i_k}$ in $X$, consider the element 
\[
 \lambda^{k-1}(x_{i_1},\dots,x_{i_k})=[[\dots[[x_{i_1}, x_{i_2}], x_{i_3}]\dots], x_{i_k}]
\]
in $L_k(m)$.
The set formed by these elements will be denoted by $\lambda^{k-1}(X^k)$. 
In general $\lambda^{k-1}(X^k)$ is not a linearly independent set, but it generates the homogeneous component $L_k(m)$. 
This is clear for $L_1(m)$ and $L_2(m)$ and follows inductively for $L_k(m)$, 
by the Jacobi identity.
\end{remark}

\begin{theorem}
The free $k$-step nilpotent Lie algebra $L_{(k)}(m)$ is rigid in $\N_{n,k}$.
\end{theorem}

\begin{proof}
According to Corollary \ref{coro:k-rig} (see the Appendix), it suffices to prove that $H_{k\textrm{-nil}}^2(\mu,\mu)=0$. 
Given $\sigma\in Z_{k\textrm{-nil}}^2(\mu,\mu)$\color{black}, we construct a linear function 
$f:L_{(k)}(m)\rightarrow L_{(k)}(m)$ such that $\sigma=\delta f$, so that $\sigma\in B_{k\textrm{-nil}}^2(\mu,\mu)$.

We start by considering the linear function $f:A(m)\rightarrow L_{(k)}(m)$
defined in the basis $M(m)$ of $A(m)$ recursively by

\begin{gather*}
f(x)=0, \text{\ for all\ } x\in X=M_1(m)\\   
f(p\mp q)=[f(p),\pi_k(q)]+[\pi_k(p),f(q)]-\s(\pi_k(p),\pi_k(q)).
\end{gather*}

A direct calculation shows that $f$ satisfies, for all $a,b,c\in A(m)$,
\begin{gather*}
f(a\mp b+b\mp a)=0, \\
f(\sc \mp^2 (a,b,c))=0.
\end{gather*}
So that it induces a linear map
$f:L(m)=A(m)/I\rightarrow L_{(k)}(m)$.
This map satisfies, for $u,v\in L(m)$, 
\begin{equation}\label{eqn:f}
  f(\lambda(u,v))=[f(u),\bar{v}]+[\bar{u},f(v)]-\s(\bar{u},\bar{v}).
\end{equation}
Recall that $\bar{\ }:L(m)\rightarrow L_{(k)}(m)$ is the quotient projection.	Moreover, we will prove that $f(L(m)^{k+1})=0$ and therefore it induces 
a linear map $f:L_{(k)}(m)\rightarrow L_{(k)}(m)$ satisfying,
for all $u,v\in L_{(k)}(m)$,
\[ f([u,v])=[f(u),v]+[u,f(v)]-\s(u,v). \]
And thus $\s=\delta f$.

It remains to see that $f(L^{k+1}(m))=0$. 
To prove this, it suffices to show that $f\circ\l^i(X^{i+1})=0$, for all $i\geq k$, because $\lambda^i(X^{i+1})$ generates $L^{i+1}(m)$
(see \ref{rmk:chains}).

That $f\circ\l^i(X^{i+1})=0$, for all $i\geq k$ is a consequence of the 
following identity of $(i+1)$-multilinear functions in $X^{i+1}$:
\begin{equation}\label{eqn:f-lambda}
 f\circ\l^{i}=
-\sum_{j=0}^{i-1}{\mu}^{i-1-j}\circ\s\circ\mu^{j},
\end{equation}
where $\mu$ is the Lie bracket of $L_{(k)}(m)$. 

We prove this identity by induction.
If $i=1$ and $(x,y)\in X^2$, since $\overline{x}=x$, $\overline{y}=y$, and $f(x)=0=f(y)$, by using \eqref{eqn:f} we have that
\begin{eqnarray*}
 f\circ\lambda^1(x,y) &=& f(\lambda(x,y)) \\
     &=& [f(x),y]+[x,f(y)] -\s(x,y) \\
     &=& -\sum_{j=0}^{0}\mu^{0-j}\circ\s\circ\mu^{j}(x,y).
\end{eqnarray*}
If $f\circ\l^{i}=-\sum_{j=0}^{i-1}\mu^{i-1-j}\circ\s\circ\mu^{j}$ and 
$\x=(\x',x)\in X^{i+2}$, for some $\x'\in X^{i+1}$ and $x\in X$,
we have that
\begin{eqnarray*}
f\circ\l^{i+1}(\x) &=& f(\lambda(\l^{i}(\x'),x)) \\
  &=& [f(\l^{i}(\x')),\overline{x}]+[\overline{\l^{i}(\x')},f(x)]-
       \s(\overline{\l^{i}(\x')},\overline{x}) \\
  &=& [f\circ\l^{i}(\x'),x]+\s(\mu^{i}(\x'),x);
\end{eqnarray*}
the second identity follows from \eqref{eqn:f} and 
the third one follows because $\overline{x}=x$, 
$\overline{\l^i(\x')}=\mu^i(\overline{\x'})=\mu^i(\x')$ and $f(x)=0$.
Now by the inductive hypothesis, we have that
\begin{eqnarray*}
f\circ\l^{i+1}(\x)
  &=& \big[-\sum_{j=0}^{i-1}\mu^{i-1-j}\circ\s\circ\mu^{j}(\x'),x\big]-   
      \s(\mu^{i}(\x'),x) \\
  &=& -\sum_{j=0}^{i-1}\mu^{i-j}\circ\s\circ\mu^{j}(\x)-\mu^{i-i}\circ 
     \s\circ  \mu^{i}(\x) \\
  &=& -\sum_{j=0}^{i}\mu^{i-j}\circ\s\circ\mu^{j}(\x).
\end{eqnarray*}     
Finally, we show that $f\circ \lambda^i=0$, for all $i\ge k$.
If $i\geq k$, then $i=k+c$, for some $c\geq0$. 
So that
\begin{align*}
f\circ\l^{i}&=-\sum_{j=0}^{i-1}\mu^{i-1-j}\circ\s\circ\mu^{j}\\
&=-\sum_{j=0}^{k+c-1}\mu^{k+c-1-j}\circ\s\circ\mu^{j}\\
&=-\sum_{j=0}^{k-1}\mu^{k+c-1-j}\circ\s\circ\mu^{j}-\sum_{j=k}^{k+c-1}\mu^{k+c-1-j}\circ\s\circ\mu^{j}\\
&=-\mu^{c}\circ(\sum_{j=0}^{k-1}\mu^{k-1-j}\circ\s\circ\mu^j)-\sum_{j=k}^{k+c-1}\mu^{k+c-1-j}\circ\s\circ\mu^{j}\\
&=0.
\end{align*}
The last identity holds because both terms are zero. 
The first one is zero, because $\s\in Z_{k\textrm{-nil}}^2(\mu,\mu)$ and hence $\sum_{j=0}^{k-1}\mu^{k-1-j}\circ\s\circ\mu^j=0$. 
The second one is equal to zero, because $\mu$ is $k$-step nilpotent Lie algebra and hence $\mu^j=0$, for all $j\geq k$.
\end{proof}

\section{Heisenberg Lie algebras}

The $(2m+1)$-dimensional Heisenberg Lie algebra is 
$\h_m=V\oplus Z=\langle x_1,y_1,\dots,x_{m},y_{m}\rangle\oplus\langle z\rangle$, where the non-zero brackets of basis elements are 
\[\mu(x_{i},y_{i})=z, \text{\ for all } 1\leq i\leq m.
\]

$\h_m$ is a $2$-step nilpotent Lie algebra with center $Z$.  Sometimes for convenience we will write $[\ ,\ ]$ instead of $\mu$ and $n$ instead of $2m+1$.

Let $\{x_1^*,y_1^*,\dots,x_m^*,y_m^*,z^*\}$ be the dual basis 
of the given one and for any $x\in\h_m$, denote by $x_V$ 
the unique element in $V$ such that $x=x_V+z^*(x)z$.

\begin{remark}\label{rmk:xv}
Given $x,y\in\h_m$, $\ad_x=\ad_y$ if and only if $x_V=y_V$
or equivalently if there is $v\in V$ such that $x=v+z^*(x)z$ and 
$y=v+z^*(y)z$.
\end{remark}

\begin{theorem}
The $2m+1$-dimensional Heisenberg Lie algebra $\h_m$ is non-rigid in $\N_{2m+1,3}$, for all $m>1$.
\end{theorem}

\begin{proof}
We will give a $3$-step linear deformation of $\h_m$ using Corollary \ref{coro:main-construction}.

Let $a_1=x_1,\, a_2=x_2,\, \h=\langle B-\{x_1,x_2\}\rangle$ and $y=y_1$. Hence, we can consider the linear deformation of $\h_m$ 
\[ [\ ,\ ]_t=[\ ,\ ] + t(a_1\wedge a_2\otimes y), \]
which is $3$-step nilpotent, for all $t\ne 0$.
\end{proof}

\begin{theorem}
The $2m+1$-dimensional Heisenberg Lie algebra $\h_m$ is rigid in $\N_{2m+1,2}$, 
for all $m\in\NN$.
\end{theorem}

\begin{proof}
According to Corollary \ref{coro:k-rig}, it is enough to prove that $H_{2\textrm{-nil}}^2(\mu,\mu)=0$. Given $\sigma\in Z_{2\textrm{-nil}}^2(\mu,\mu)$, we will prove that there exist a linear function $f:\h_m\rightarrow \h_m$, such that 
$\sigma=\delta f$.

Since $[\ ,\ ]\circ\s=-\s\circ[\ ,\ ]$,  on the one hand
\begin{equation}\label{eqn:sigma}
 \ad_{\s(x_i,y_i)}=-\s(z,\cdot)=\ad_{\s(x_j,y_j)},
\end{equation}
 for all $i,j\in\{1,...,m\}$,
and hence there exists $v\in V$ such that 
$$\s(x_i,y_i)=v+z^*(\s(x_i,y_i))z,$$
for all $i=1,...,m$ (see Remark \ref{rmk:xv}).
On the other hand
\begin{equation}
\text{if } [x,y]=0, \text{then } \s(x,y)\in Z.
\end{equation} 
Now, $\delta f=\s$ if and only if $f$ satisfies the following system of linear equations:
\begin{alignat*}{3}\label{}
&\s(x_{i},y_{i}) &=& [f(x_{i}),y_{i}]+[x_{i},f(y_{i})]-f(z), 
                &\quad&\text{for $i=1,\dots,m$}; \\
&\s(x_i,y_j) &=& [f(x_i),y_j]+[x_i,f(y_j)],  
                &&\text{for $1\le i,j\le m, i\ne j$}; \\
&\s(x_{i},x_{j}) &=& [f(x_{i}),x_{j}]+[x_{i},f(x_{j})], 
                && \text{for $1\le i<j\le m$ }; \\
&\s(y_{i},y_{j}) &=& [f(y_{i}),y_{j}]+[y_{i},f(y_{j})], 
&& \text{for $1\le i<j\le m$ }; \\
&\s(z,x_{i}) &=& [f(z),x_{i}], 
                && \text{for $i=1,\dots,m$ }; \\
&\s(z,y_i) &=& [f(z),y_i], 
                && \text{for $i=1,\dots,m$}.
\end{alignat*}
We start by defining $f(z)=-v$.
Then the last two sets of equations are satisfied
because $\sigma$ satisfies \eqref{eqn:sigma}. 

To define $f$ in $V$, we set that
$z^*(f(v))=0$, for all $v\in V$, so that
\begin{eqnarray*}
f(x_i) &=& \sum_{k=1}^{m} x_k^*(f(x_i))x_k+\sum_{k=1}^{m} y_k^*(f(x_i))y_k, \\
f(y_i) &=& \sum_{k=1}^{m} x_k^*(f(y_i))x_k+\sum_{k=1}^{m} y_k^*(f(y_i))y_k.
\end{eqnarray*}
It remains to determine 
$a_k^i=x_k^*(f(x_i))$, $b_k^i=y_k^*(f(x_i))$, $c_k^i=x_k^*(f(y_i))$ and $d_k^i=y_k^*(f(y_i))$, for $k,i\in \{1,\dots,m\}$, which now must satisfy
the system of linear equations
\begin{alignat*}{3}
\s(x_{i},y_{i})-v &= (a_i^i+d_i^i)z, 
                        &\qquad&\text{for $i=1,\dots,m$};\\
\s(x_i,y_j) &= (a_j^i+d_i^j)z,
                        &&\text{for $1\le i,j\le m, i\ne j$}; \\
\s(x_{i},x_{j}) &= (-b_j^i+b_i^j)z, 
                        && \text{for $1\le i,j\le m$ }; \\
\s(y_{i},y_{j}) &= (c_j^i-c_i^j)z, 
                        && \text{for $1\le i,j\le m$ }. \\
\end{alignat*}
Since this one is clearly consistent, we are done.
\end{proof}

\begin{remark}
 It is worth mentioning that the same result can be found in \cite{GR2}.
\end{remark}

\section{Lie algebras with an abelian factor}

In this section, we explore the rigidity and non-rigidity of Lie algebras with an abelian factor.

Given a Lie algebra $(\g,\lambda)$ and an abelian Lie algebra $(\a,\mu)$,
the question we address in this section is whether is
$\g\oplus\a$ rigid or not ($\g\oplus\a$ with the Lie bracket given by the direct sum of Lie algebras).
Denote by $\lambda\oplus\mu$ the bracket of $\g\oplus\a$.

The answer, which depends on the size of $\a$ and the framework variety,
is in general no.
However the situation for small abelian factors, in particular for one
dimensional factors are quite interesting.

It is worth recalling that for Lie algebras with an abelian ideal or subalgebra that is not an abelian factor, the situation is different.
In fact in \cite{R} rigid Lie algebras of the form $\g\ltimes\a$ have been constructed
and in \cite{AG} rigid solvable Lie algebras of the form $\g\rtimes\a$ have been constructed.

Given an $m$-dimensional Lie algebra $(\g,\lambda)$ and the $l$-dimensional abelian Lie algebra, denoted by $(\a_l,\mu)$,
we observe that for any non-abelian Lie bracket $\nu$, \begin{equation}\label{eq:def-ab}
\mu_t=t\nu
\end{equation}
is a non-trivial linear deformation of $\mu$. This
gives rise \color{black}to the non-trivial deformation
\[ (\lambda\oplus\mu)_t=\lambda\oplus t\nu, \]
of the $m+l$-dimensional Lie algebra $\lambda\oplus\mu$. 
From this observation it follows that:
\begin{enumerate}[(1)]
 \item If $l\ge 2$, then for any $\g$, $\g\oplus\a_l$ is non-rigid in $\L_{m+l}$.
 
 \item If $l\ge 2$, then for any $m$-dimensional solvable Lie algebra $\mathfrak{s}$, 
 $\mathfrak{s}\oplus\a_l$ is non-rigid in $\S_{m+l}$.
 Moreover, if $\mathfrak{s}$ is $k$-step solvable, 
 $\mathfrak{s}\oplus\a_l$ is non-rigid in $\S_{m+l,k}$.
 
 \item If $l\ge 3$, then for any $m$-dimensional nilpotent Lie algebra $\n$, $\n\oplus\a_l$ 
 is non-rigid in $\N_{m+l}$.
 Moreover, if $\n$ is $k$-step nilpotent, 
 $\n\oplus\a_l$ is non-rigid in $\N_{m+l,k}$.
\end{enumerate}
Hence, the only cases remaining to consider are:
\begin{enumerate}[(1')]
 \item $\g\oplus\a_l\in\L_{m+l}$, where $\g$ is any Lie algebra and $l=1$.
  
 \item $\mathfrak{s}\oplus\a_l\in\S_{m+l}$, where $\mathfrak{s}$ is any $k$-step solvable Lie algebra
 and $l=1$.
 And its stronger version for $\mathfrak{s}\oplus\a_l\in\S_{m+l,k}$.
  
 \item $\n\oplus\a_l\in\N_{m+l}$, where $\n$ is any $k$-step nilpotent Lie algebra and $l=1$ or
 $l=2$. 
 And its stronger version for $\n\oplus\a_l\in\N_{m+l,k}$.
\end{enumerate}

We will show that the answer for (2') and for (3') is the general one,
namely, they are not rigid, even in their strongest forms.
The answer for (1') turns out to be more intricate.
Even though there is no unified answer independently of $\g$,
we shall see that being $\g$ perfect or not plays a role.
Recall that $\g$ is called \emph{perfect} if $\lambda(\g,\g)=\g$.

\begin{example}\label{ex:semisimple}
Given $\g$ a $m$-dimensional semisimple Lie algebra, let $\overline{\g}=\g\oplus\a_1$. 
Then $H^2(\overline{\g},\overline{\g})=0$ and hence $\overline{\g}$ 
is rigid in $\L_{m+1}$.
The fact that the second adjoint cohomology group of $\overline{\g}$ 
vanishes follows directly from the Hochschild-Serre spectral sequence
associated to the ideal $\g$ of $\overline{\g}$ and the fact that
$H^1(\g,\g)=H^2(\g,\g)=0$.

\end{example}

We address now questions (1'), (2') and (3'), where the abelian
factor is small, one at a time.
For the sake of completeness, the statements are given for arbitrary
abelian factors.

\begin{notation}
	For convenience, we will denote by $[\ ,\ ]$ the bracket of 
	$\g\oplus\a_l$.
\end{notation}

\subsection{Non-perfect Lie algebras}

\begin{theorem}\label{thm:non-perfect}
If $\g\in\L_{m}$ is non-perfect, then $\g\oplus\a_l$ is non-rigid 
in $\L_{m+l}$, for all $l\in\NN$.
\end{theorem}

\begin{proof}
The only remaining case is $l=1$.
We will give a linear deformation of $\g\oplus\a_1$ in $\L_{n+l}$, using Theorem \ref{thm:main-construction}.

Given $A=\{a\}$ any basis of $\a_1$ and $B$ any basis of $\g$, there exists $b\in B$ such that $b\notin[\g,\g]$. 
Take $a_1=a, a_2=b, \h=\langle (A\cup B)-\{a,b\}\rangle$ and $y=a$.  So that, from the Corollary \ref{coro:main-construction}, we can consider
the linear deformation of $\g\oplus\a_1$ 
\[ [\ ,\ ]_t=[\ ,\ ] + t(a_1\wedge a_2\otimes y). \]
This deformation is non-trivial, since the dimension of the commutator corresponding to $t\ne 0$ is larger than the dimension of the original one. 
\end{proof}

\subsection{Solvable Lie algebras}

\begin{theorem}
If $\mathfrak{s}\in\S_{m,k}$, then $\mathfrak{s}\oplus\a_l$ is 
non-rigid in $\S_{m+l,k}$, for all $l\in\NN$.
\end{theorem}	

\begin{proof}
The only remaining case is $l=1$.
Since $\mathfrak{s}$ is solvable, it is non-perfect. 
The deformation given in the proof of the Theorem \ref{thm:non-perfect} is a $k$-step solvable linear deformation. 
\end{proof}

\subsection{Nilpotent Lie algebras}

In this section, we complete the proof of the fact that
any $m$-dimensional $k$-step nilpotent Lie algebra plus the $l$-dimensional abelian Lie algebra is not rigid,
even in the smaller subvariety $\N_{m+l,k}$,
by considering the remaining cases where the dimension of
the abelian factor is $l=2$ or $l=1$.
It is worth saying that there is a single exception, 
namely $\h_1\oplus\a_1$.

\begin{remark}
 There are only two, non isomorphic, nilpotent Lie algebras 
 of dimension 4 in $\N_{4,2}$: $\a_4$ and $\h_1\oplus\a_1$.
 Therefore, $\h_1\oplus\a_1$ is rigid in $\N_2$.
 It is worth mentioning that $\h_1\oplus\a_1$ is not rigid in 
 $\N_{4,3}=\N_4$.
\end{remark}

Nilpotent Lie algebras are non-perfect,
so by Theorem \ref{thm:non-perfect}, 
$\n\oplus\a_l$ is non-rigid in $\L_{m+l}$, for all $l\in\NN$. 
But to prove that $\n\oplus\a_l$ is non-rigid in $\N_{m+l,k}$ we must work harder, since the deformation constructed in the proof of that
theorem is not nilpotent.

\medskip

\begin{notation}\label{notation}
We fix some notation for what follows.
The bases for $\a_1$ and $\a_2$ will be $A=\{c_1\}$ and $A=\{c_1,c_2\}$
respectively. 
For $\n\in\N_{m,k}$, we choose $B_k\subseteq...\subseteq B_1$ such that $B_i$ is a basis of $\n^i$, the $i$-th term of the descending central series 
of $\n$, for all $i=1,\dots,k$. 
We denote $B_i-B_{i+1}=\{x_1^i,\dots,x_{n_i}^i\}$, for $i=1,\dots,k-1$. 
Notice that $B=B_1$ is a basis of $\n$ 
and $B\cup A$ is a basis of $\n\oplus\a_l$. 
\end{notation}

\begin{proposition}\label{prop:a2}
 If $\n\in\N_{m,k}$, then
 $\n\oplus\a_2$ is non-rigid in $\N_{m+2,k}$.
\end{proposition}

\begin{proof}
We will construct a non-trivial linear deformation of $\n\oplus\a_2$ in $\N_k$
using Corollary \ref{coro:main-construction}. 

Take $a_1=c_1$, $a_2=x_1^1$, $\h=\langle (A\cup B)-\{c_1,x_1^1\}\rangle$ and $y=c_2$.  
So that, we can consider
the linear deformation of $\n\oplus\a_2$ given by
\[ [\ ,\ ]_t=[\ ,\ ] + t(a_1\wedge a_2\otimes y). \]
It is easy to see that this deformation is $k$-step nilpotent. 
It is non-trivial, since the dimension of the commutator corresponding to $t\ne 0$ is larger than the dimension of the original one. 
\end{proof}

We come now to the most difficult case, that for $l=1$. 
We look at the 2-step nilpotent quotient
$\tilde\n=\dfrac{\n}{\n^3}$; recall that $\n^3=\lambda(\lambda(\n,\n),\n)$.
And we split the proof into two propositions, according to whether
this quotient is isomorphic to a free 2-step nilpotent Lie algebra or it is not. 

In general, given a nilpotent Lie algebra $\g$
with an adapted basis $B$ as in \ref{notation}, by taking  
$a_1,a_2\in B_1-B_2$, $\h=\langle B-\{a_1,a_2\}\rangle$ and 
$y$ a central element, the hypotheses of our construction are trivially fulfilled. 
In the case we are dealing with, in which $\g=\n\oplus\a_1$, 
we may take $y=c_1$.
Assuming that $\tilde\n\ncong L_{(2)}(m)$ we are able to prove that the resulting deformation is non-trivial.
In the case $\tilde\n\cong L_{(2)}(m)$, we do something different.

\begin{proposition}
If $\n\in\N_{m,k}$ and $\tilde\n\ncong L_2(m)$, then $\n\oplus \a_1$ is non-rigid in $\N_{m+1,k}$.
\end{proposition}

\begin{proof}
We will construct a non-trivial linear deformation of $\n\oplus\a_1$ in $\N_{m+1,k}$, 
using Corollary \ref{coro:main-construction}. 

Since $\displaystyle n_2< \genfrac(){0pt}{2}{n-1}{2}$, 
the projection of the set $\{[x_j^1,x_i^1]: 1\leq i<j\leq n_1\}$ 
onto $\tilde\n$ is a linearly dependent set. 
Hence, after relabeling if necessary, we may assume that 
\begin{equation}\label{eq:notriv}
[x_2^1,x_1^1]\in\left\langle{\Big(\big\{[x_j^1,x_i^1]: 1\leq i<j
\leq n_1\}-\{[x_2^1,x_1^1]\big\}\Big)\cup B_3}\right\rangle.
\end{equation}
Taking $a_1=x_1^1, a_2=x_2^1, \h=\langle A\cup( B-\{x_1^1,x_2^1\}\rangle$ and $y=c_1$, from Corollary \ref{coro:main-construction}, we can consider
the linear deformation of $\n\oplus\a_1$ given by
\[ [\ ,\ ]_t=[\ ,\ ] + t(a_1\wedge a_2\otimes y). \]
It is easy to see that this deformation is $k$-step nilpotent. Also it is 
non-trivial because the dimension of the commutator corresponding to $t\ne 0$ is larger than the dimension of the original one (\ref{eq:notriv}). 
\end{proof}

For the last case, we shall assume without lost of generality,
that $\n$ has not abelian factor. In fact, if $\n$ has an abelian
factor then $\n\oplus\a_1$ falls in the case covered by 
Proposition \ref{prop:a2}.

\begin{proposition}
If $\n\in\N_{m,k}$ has not abelian factor, 
$\n\ncong \h_1$ and $\tilde\n\cong L_{(2)}(m)$, 
then $\n\oplus\a_1$ is non-rigid in $\N_{m+1,k}$.
\end{proposition}

\begin{proof}
We will construct a non-trivial linear deformation of $\n\oplus\a_1$ in $\N_{m+1,k}$, 
using Corollary \ref{coro:main-construction}. 

Consider the sets $S$ and $R$,
\begin{align*}
&S=\big\{x\in\n:\, [x,\n^2]=0 \text{ and }\dim([x,\n])\leq 1\big\},\\ 
&R=\big\{r\in\{1,\dots,k\}:\, S\cap (\n^r-\n^{r+1})\neq\emptyset\big\}.
\end{align*}
Let $r_0=\min R$.
Notice that if $r_0=1$, then $m=2$. If $r_0\geq 2$, there exists $y_0\neq0$ such that $y_0\in S\cap\n^2$. 
Let us consider separately the cases $r_0=1$
and $r_0\ge 2$.
\smallskip 

\noindent {\bf Case 1:} If $r_0=1$, then $\tilde{\n}\cong L_{(2)}(2)$. Since $\n\ncong L_{(2)}(2)\simeq\h_1$, $k\geq 3$. 
We may choose 
$$B=\{x_1^1,x_2^1\}\cup\{[x_1^1,x_2^1]\}\cup B_3,$$
with $x_2^1\in S$.
We take $a_1=x_2^1$, $a_2=[x_1^1,x_2^1]$, $\h=\langle\{x_1^1\}\cup B_3\rangle$ and $y=c_1$.  Hence, using Corollary \ref{coro:main-construction}, we can consider
the linear deformation of $\n\oplus\a_1$ 
\[ [\ ,\ ]_t=[\ ,\ ] + t(a_1\wedge a_2\otimes y). \]
It is easy to see that this deformation is $k$-step nilpotent. Also it is 
non-trivial because $\dim((\n\oplus\a_1)_t^2)=\dim((\n\oplus\a_1)^2)+1$,
for $t\ne 0$.
\smallskip 

\noindent {\bf Case 2:} If $r_0\geq 2$, let $0\neq y_0\in S\cap(\n^{r_0}-\n^{r_0+1})$. Choose $B$ such that
\begin{align*}
&[ y_0,b]=0,\text{\ for all\ } b\in B-\{x_1^1\},
\end{align*} 
and take $a_1=x_1^1$, $a_2=c_1$, $\h=\langle B-\{x_1^1\}\rangle$ and $y=y_0$.  So that, from Corollary \ref{coro:main-construction}, we can consider
the linear deformation of $\n\oplus\a_1$ given by
\[ [\ ,\ ]_t=[\ ,\ ] + t(a_1\wedge a_2\otimes y). \]
This deformation is $k$-step nilpotent, because $y_0\in\n^2$.

In order to prove that this deformation is non-trivial,
assume instead that for arbitrary small $t$, 
$(\n\oplus\a_1)_t\cong\n\oplus\a_1$. 
Hence $(\n\oplus\a_1)_t$ has an abelian factor $\langle z\rangle$. 
Then
 \begin{gather}
[z, b]_t=0, \text{\ for all\ } b\in B; \label{eqn:z}\\
z\notin (\n\oplus\a_1)_t^2=\n^2. \label{eqn:z2}
\end{gather}
Writing $z=z_\n+\alpha c_1$, \eqref{eqn:z} implies that
\begin{gather}
[z_n, b]=0, \text{ for all } b\in B-\{x_1^1\}; \label{eqn:zn1} \\
[z_n, x_1^1]=[z_n, x_1^1]_t=t \alpha y_0. \label{eqn:zn2} 
\end{gather}
On the one hand, if $\alpha=0$, then $z_n\in Z(\n)$ and $z_\n\notin\n^2$
\eqref{eqn:z2} and therefore $z_\n$ is an abelian factor of $\n$.
On the other hand, if $\alpha\neq 0$, $z_\n\in S$ \eqref{eqn:zn1}
and then by \eqref{eqn:zn2} there exists $r\in R$ with $r<r_0$.
\end{proof}

Summarizing all we have proved, it follows that $k$-step nilpotent Lie algebras with an abelian factor are never $k$-rigid,
except for $\h_1\oplus\a_1$.

\begin{theorem}
	If $\n\in\N_{m,k}$ and $l\geq 1$, then $\n\oplus\a_l$ is rigid in $\N_{m+l,k}$
	if and only if $\n\cong \h_1$ and $l=1$.
\end{theorem}

\subsection{The exceptional case}

\ 

The only exceptional case for which there is no unified answer on whether 
$\g\oplus\a_l$ is rigid or not in $\L_{m+l}$,
is for $\g$ a perfect Lie algebra and $l=1$.
Example \ref{ex:semisimple} shows that the answer might be ``rigid''.
The following example shows that the answer might be ``non-rigid''.

\begin{example}
Let $\g$ be the complex 5-dimensional Lie algebra with basis $\{a,b,c,d,e\}$ and
bracket defined by:
\[
\begin{gathered}{}
\lambda(a,b)=2b, \qquad \lambda(a,c)=-2c, \qquad \lambda(b,c)=a, \\
\lambda(a,d)=d, \qquad \lambda(a,e)=-e, \qquad \lambda(b,e)=d, \qquad \lambda(c,d)=e.
\end{gathered}
\] 
Notice that $\g=\sl_2\ltimes \C^2$ where the semidirect product is given by the $2$-dimensional 
irreducible representation of $\sl_2$.
It holds that $H^2(\g,\g)=0$ and hence $\g$ is rigid.

Let $\overline{\g}=\g\oplus\a_1$ where $\a_1$ is an abelian factor,
let its Lie bracket be denoted by $[\ ,\ ]$ and let $\{f\}$ be a basis for $\a_1$.
Consider the linear deformation of $\overline{\g}$ given by
$$[\ ,\ ]_t=[\ ,\ ]+t\varphi,$$
where $\varphi$ is the 2-cocycle
\[ \varphi:= d^*\wedge e^* \otimes f. \]
Then $[\ ,\ ]_t$ is given by
\[
\begin{gathered}{}
[a,b]_t=2b, \qquad [a,c]_t=-2c, \qquad [b,c]_t=a, \\
[a,d]_t=d, \qquad [a,e]_t=-e, \qquad [b,e]_t=d, \qquad [c,d]_t=e, 
\qquad [d,e]_t=tf. 
\end{gathered}
\]
Clearly $\overline{\g}_t$ is perfect, for every $t\ne 0$, so that
it has no abelian factor and the deformation is non-trivial.
Thus $\overline{\g}$ is non-rigid.

\end{example}

\section{Appendix}
\smallskip
\begin{center}by \textsc{Diego Sulca}\end{center}
\medskip

The classical Nijenhuis-Richardson theorem asserts that an $n$-dimensional Lie algebra $\g$ for which the 
second Cartan-Eilenberg cohomology $H^2(\g,\g)$ is zero must be rigid in the variety of $n$-dimensional Lie algebras \cite{NR}. 
The proof given in \cite{NR} can be easily adapted to show analogous results for other classes of algebras. 
The general strategy is discussed by Remm in \cite[Sections 2.2 and 2.3]{Remm}. 
We provide full details and apply this generalization to the variety of $n$-dimensional $k$-step nilpotent Lie algebras 
and the variety of $n$-dimensional $k$-step solvable Lie algebras. 
We make use of the language of schemes.
For the information of algebraic groups acting on schemes, we refer to \cite[Chapter 7]{Milne}. 
Throughout, $\K$ denotes any field of characteristic zero. 

\

\noindent Let $\mathbb{A}_\K^m$ be the affine $m$-space over $\K$ and let $X\subset \mathbb{A}_\K^m$ be closed subscheme. 
Given a rational point $x\in X(\K)$ the Zariski tangent space $T_x X$ of $X$ at $x$ can be computed as
$$T_xX=\{y\in \K^m: x+\varepsilon y\in X(\K[\varepsilon])\},$$
where $\K[\varepsilon]=\K+\K\varepsilon$ is the $\K$-algebra of dual numbers ($\varepsilon^2=0$). 
We have $\dim_x X\leq \dim T_x X$, where $\dim_x X$ is the local dimension of $X$ at $x$ 
(i.e., the dimension of the local ring of $X$ at $x$) and $\dim T_x X$ is the dimension of $T_x X$ as vector space over $\K$. 
The equality holds if and only if $x$ is a non-singular point of $X$. 

\

\noindent Fix now  $n\in\mathbb{N}$. We think of $\mathbb{A}_{\K}^{n^3}$ as representing the functor
$$
R\mapsto \{R\mbox{-bilinear maps}\ R^n\times R^n\to R^n\}=\{\K\mbox{-bilinear maps}\ \K^n\times \K^n\to R^n\}
$$
from commutative $\K$-algebras to the category of sets.

The linear group $GL_n$ (viewed as affine group scheme over $\K$) acts on $\mathbb{A}_{\K}^{n^3}$ as follows: 
given a commutative $\K$-algebra $R$, a matrix $g\in GL_n(R)$ and an $R$-bilinear map $\mu:R^n\times R^n\to R^n$, we define $g\cdot\mu:R^n\times R^n\to R^n$ by setting 
$$(g\cdot \mu)(x,y):=g(\mu(g^{-1}(x),g^{-1}(y))),\quad x,y\in R^n.$$ 
Let ${X}\subset\mathbb{A}_{\K}^{n^3}$ be a closed subscheme that is invariant under the action of $GL_n$.
Fix a rational point $\mu\in {X}(\K)$ (if there are any). 
The image of the orbit map $GL_n\to {X}$, $g\to g\cdot\mu$, is locally closed in ${X}$. The {\em orbit $O(\mu)$ of $\mu$} is this image equipped with its structure of a reduced subscheme of ${X}$. It is smooth over $\K$.
The {\em isotropy group $G_\mu$ at $\mu$} is a closed subgroup of $GL_n$, and for all commutative $\K$-algebras $R$,
$$G_\mu(R)=\{g\in GL_n(R): g\cdot\mu_R=\mu_R\}$$
where $\mu_R\in {X}(R)$ denotes the image of $\mu\in {X}(\K)$ in ${X}(R)$.
The orbit map $GL_n\to O(\mu)$ induces a $\K$-isomorphism 
\begin{align}\label{the orbit as a quotient}
GL_n/G_\mu\cong O(\mu).
\end{align}

\

\noindent For $\mu\in {X}(\K)$ we define
\begin{align*}
Z_{{X}}^2(\mu,\mu)&:=T_\mu({X})=\{\omega:\K^n\times\K^n\to\K^n\ | \ \mu+\varepsilon\omega\in {X}(\K[\varepsilon])\},\\
B^2(\mu,\mu)&:=\{\delta g:\K^n\times\K^n\to \K^n\ |\ g\in \mathfrak{gl}_n(\K)\},
\end{align*}
where for $g\in\mathfrak{gl}_n(\K)$, $\delta g$ is the $\K$-bilinear map
$$\delta g(x,y)=g(\mu(x,y))-\mu(g(x),y)-\mu(x,g(y)), \ x,y\in\K^n.$$
Notice that 
$$B^2(\mu,\mu)\subseteq Z_{{X}}^2(\mu,\mu).$$
Indeed, given $g\in\mathfrak{gl}_n(\K)$, we have $I+\varepsilon g\in GL_n(\K[\varepsilon])$, and it is easy to check that 
\begin{align}\label{useful identity}
\mu+\varepsilon \delta g=(I+\varepsilon g)\cdot \mu_{K[\varepsilon]}\in {X}(\K[\varepsilon]),
\end{align}
which shows that $\delta g\in Z_X^2(\mu,\mu)$.
Finally, set
\begin{align*}
H_{{X}}^2(\mu,\mu):=\frac{Z_{X}^2(\mu,\mu)}{B^2(\mu,\mu)}.
\end{align*}
\begin{theorem}\label{second cohomology zero implies rigidity}
Let $\K$ be a field of characteristic zero.
For $\mu\in \mathcal(\K)$ the following conditions are equivalent.
\begin{enumerate}
  \item  $H_{X}^2(\mu,\mu)=0$.
  \item  $O(\mu)$ is an open subscheme of ${X}$. 
  \item  $O(\mu)$ is an open subset of $X$ and ${X}$ is reduced at $\mu$.
\end{enumerate}
\end{theorem}
\begin{proof} We first review the fact that $T_\mu O(\mu)$ is isomorphic to $B^2(\mu,\mu)$. 
	By (\ref{the orbit as a quotient}), there are $\K$-linear isomorphisms
	$$T_\mu O(\mu)\cong T_e(GL_n/G_\mu)\cong\mathfrak{gl}_n(\K)/\on{Lie}(G_\mu)$$
	where $e\in GL_n/G_\mu$ denotes the image of the identity of $GL_n$.
	Thus, it is enough to show that $\on{Lie}(G_\mu)$ is the kernel of the surjective map $\delta:\mathfrak{gl}_n(\K)\to B^2(\mu,\mu)$. 
	Now
	\begin{align*} 
	\on{Lie}(G_\mu)&=\{g\in \mathfrak{gl}_n(\K): I+\varepsilon g\in G_\mu(\K[\varepsilon])\}\\
	&=\{g\in \mathfrak{gl}_n(\K): (I+\varepsilon g)\cdot \mu_{K[\varepsilon]}=\mu_{K[\varepsilon]}\}
	\end{align*}
	As observed in (\ref{useful identity}) we have $(I+\varepsilon g) \mu_{K[\varepsilon]}=\mu+\varepsilon\delta g$. Hence, 
	$(I+\varepsilon g)\cdot \mu_{K[\varepsilon]}=\mu_{K[\varepsilon]}$ if and only if $\delta g=0$, as was to be shown.

	We now proceed with the proof of the equivalences. Note first that
	\begin{align}\label{ineq main theorem}
	\dim B^2(\mu,\mu)&=\dim T_\mu O(\mu)=\dim_\mu O(\mu)\leq\dim_\mu {X}\leq\\
	\nonumber&\leq\dim T_\mu{X}_{\textrm{red}}\leq \dim T_\mu {X}=\dim Z_{{X}}^2(\mu,\mu),
	\end{align}
	where in the first equality we use the isomorphism of the above paragraph and in the second one we use the fact that $O(\mu)$ is smooth. The inequalities are clear.
	
	If $H_{{X}}^2(\mu,\mu)=0$ then all the inequalities in (\ref{ineq main theorem}) become equalities. The equality $\dim_\mu {X}=\dim T_\mu {X}$ implies that $\mu$ is a non-singular point of ${X}$. Since the set of non-singular points in a scheme of finite type over a field is open, the schemes ${X}$ and $O(\mu)$ are regular of the same dimension at a neighborhood of $\mu$. As $O(\mu)$ is a locally closed subscheme of ${X}$, we deduce that $O(\mu)$ and ${X}$ coincide locally at $\mu$.  
	Given another rational point $\mu'\in O(\mu)(\K)$, clearly $O(\mu')=O(\mu)$ and $H_{{X}}^2(\mu',\mu')\cong H_{{X}}^2(\mu,\mu)$, which is assumed to be zero. By applying the above reasoning to each such $\mu'\in O(\mu)(\K)$ we find that ${X}$ and $O(\mu)$ coincide as schemes locally at each rational point of $O(\mu)$. Now, as $\K$ is infinite, $GL_n(\K)$ is dense in $GL_n$ hence $O(\mu)(\K)$ is dense in $O(\mu)$. It follows that $O(\mu)$ is an open subscheme of ${X}$. This completes the proof of (1)$\Rightarrow$(2).
	
	(2)$\Rightarrow$(3) is obvious. We finally show that (3)$\Rightarrow$(1). The first hypothesis implies that $O(\mu)$ is an open subscheme of ${X}_{\textrm{red}}$, hence the first two inequalities of (\ref{ineq main theorem}) are indeed equalities. Since in addition ${X}$ is reduced at $\mu$, the last inequality is also an equality. Summarizing, all the inequalities in (\ref{ineq main theorem}) are equalities, hence $H_{{X}}^2(\mu,\mu)=0$.
\end{proof}


\

\noindent We compute $Z_X^2(\mu,\mu)$ for three examples of $X$.
\subsection*{The scheme $L_n$ of Lie brackets on $\K^n$ \cite{NR}}
Let $L_n\subset\mathbb{A}_\K^{n^3}$ be the closed subscheme such that for all commutative $\K$-algebras $R$,
$L_n(R)$ is the set of $R$-bilinear maps $\mu: R^n\times R^n\to R^n$ that are alternating (i.e., $\mu(x,x)=0$ for all $x\in\mathbb{R}$) and
satisfy the Jacobi identity
\begin{align*}
\sc\mu(\mu(x,y),z)=0,\quad \forall x,y,z\in R^n.
\end{align*}

Given $\mu\in L_n(\K)$, by definition $Z_{L_n}(\mu,\mu)$ is the set of bilinear maps $\omega:\K^n\times\K^n\to\K^n$ 
such that $\mu+\varepsilon\omega$ is alternating and satisfies the Jacobi identity. 
In other words, for all $x,y,z\in \K^n$,
\begin{align*}
(\mu+\varepsilon\omega)(x,x)&=0,\\
\sc \mu(\mu(x,y)+\varepsilon\omega(x,y),z)+\sc\varepsilon\omega(\mu(x,y)+\varepsilon \omega(x,y),z)&=0.
\end{align*}
As $\mu$ is alternating, the first equality is equivalent to saying that $\omega$ is alternating.
Since $\varepsilon^2=0$ and $\mu$ satisfies the Jacobi identity, the left hand side of the second equality is simply $\varepsilon \delta\omega(x,y,z)$, where
$$
\delta\omega(x,y,z):=\sc \mu(\omega(x,y),z)+\sc \omega(\mu(x,y),z)
$$
We conclude that
\begin{align*}
Z_{L_n}^2(\mu,\mu)=\{\omega:\K^n\times \K^n\to \K^n\ |\ \omega\ \mbox{is } \K\mbox{-bilinear, alternating and } \delta\omega=0\}.
\end{align*}
It follows that $H_{L_n}^2(\mu,\mu)$ is the usual Cartan-Eilenberg cohomology of the Lie algebra $(\K^n,\mu)$, denoted simply by $H^2(\mu,\mu)$.

\begin{corollary}\label{coro:Lie-rig}
Let $\K$ be an algebraically closed field of characteristic zero and let $\L_{n}=L_{n}(\K)$, 
with the structure of affine algebraic variety. 
Given $\mu\in \L_{n}$, if $H^2(\mu,\mu)=0$, then $\mu$ is rigid in $\L_{n}$. 	
\end{corollary}	

\subsection*{The scheme of $k$-solvable Lie brackets on $\K^n$} Let $S_{n,k}\subset L_n$ be the closed subscheme such that for all commutative $\K$-algebras $R$,
$S_{n,k}(R)$ is the set of those $\mu\in L_n(R)$ such that
$$
\mu^{(k)}(x_1,\ldots,x_{2^k})=0,\quad\forall x_1,\ldots,x_{2^k}\in R^n,
$$
where $\mu^{(i)}:\underbrace{R^n\times\cdots\times R^n}_{2^i}$ is defined inductively by setting $\mu^{(0)}=\on{id}_{R^n}$ and 
$$\mu^{(i)}(x_1,\ldots,x_{2^i})=\mu(\mu^{(i-1)}(x_1,\ldots,x_{2^{i-1}}),\mu^{(i-1)}(x_{2^{i-1}+1},\ldots,x_{2^i}))
$$
for $i\geq 1$.

Given $\mu\in S_{n,k}(\K)$, by definition $Z_{S_{n,k}}^2(\mu,\mu)$ is the set of those $\omega\in Z_{L_n}^2(\mu,\mu)$ satisfying the additional condition $(\mu+\varepsilon \omega)^{(k)}=0$. One easily checks that 
\begin{align*}
(\mu+\varepsilon \omega)^{(k)}(x_1,\ldots,x_{2^i})=\mu^{(k)}(x_1,\ldots,x_{2^i})+\varepsilon \sigma_k\omega(x_1,\ldots,x_{2^i}), 
\end{align*}
where $\sigma_i\omega:\underbrace{R^n\times\cdots\times R^n}_{2^i}\to R^n$ is defined inductively as follows: $\sigma_1\omega=\omega$, and
\begin{align*}
\sigma_i\omega(x_1,\ldots,x_{2^i})=&\mu(\mu^{(i-1)}(x_1,\ldots,x_{2^{i-1}}),\sigma_{i-1}\omega(x_{2^{i-1}+1},\ldots,x_{2^i}))\\
&+\mu(\sigma_{i-1}\omega(x_1,\ldots,x_{2^{i-1}}),\mu^{(i-1)}(x_{2^{i-1}+1},\ldots,x_{2^i}))\\
&+\omega(\mu^{(i-1)}(x_1,\ldots,x_{2^{i-1}}),\mu^{(i-1)}(x_{2^{i-1}+1},\ldots,x_{2^i})),
\end{align*}
for $i\geq 2$. 
Since $\mu^{(k)}(x_1,\ldots,x_{2^i})=0$ we obtain that
\begin{align*}
Z_{S_{n,k}}^2(\mu,\mu)=\{\omega\in Z_{L_n}^2(\mu,\mu)\ | \ \sigma_k\omega =0\}.
\end{align*}
We shall use the notation $H_{k{\textit{-sol}}}^2(\mu,\mu):=H_{S_{n,k}}^2(\mu,\mu)$.

\begin{corollary}\label{coro:k-sol-rig}
Let $\K$ be an algebraically closed field of characteristic zero and let $\S_{n,k}=S_{n,k}(\K)$, 
with the structure of affine algebraic variety. Given $\mu\in \S_{n,k}$, if $H^2_{k{\textit{-sol}}}(\mu,\mu)=0$, 
then $\mu$ is rigid in $\S_{n,k}$. 	
\end{corollary}

\subsection*{The scheme of $k$-step nilpotent Lie brackets on $\K^n$}  
Let $1\leq k\leq n-1$ and let $N_{n,k}\subset L_n$ be the closed subscheme such that for all commutative $\K$-algebras $R$,
$N_{n,k}(R)$ is the set of those $\mu\in L_n(R)$ such that
$$
\mu^{k}(x_1,\ldots,x_{k+1})=0,\quad\forall x_1,\ldots,x_{k+1}\in R^n,
$$
where $\mu^{i}:\underbrace{R^n\times\cdots\times R^n}_{i+1}\to R^n$ is defined inductively by setting $\mu^0=\on{id}_{R^n}$ and 
\begin{align*}
\mu^{i}(x_1,\ldots,x_{i+1}):=\mu(\mu^{i-1}(x_1,\ldots,x_i),x_{i+1})\quad\mbox{for}\quad i\geq 1.
\end{align*}

Given $\mu\in N_{n,k}(\K)$, by definition $Z_{N_{n,k}}^2(\mu,\mu)$ is the set of those $\omega\in Z_{L_n}^2(\mu,\mu)$ satisfying the additional condition $(\mu+\varepsilon\omega)^k(x_1,\ldots,x_{k+1})=0$ for all $x_1,\ldots,x_{k+1}\in \K^n.$
One easily checks by induction that
\begin{align*}
(\mu+\varepsilon\omega)^k(x_1,\ldots,x_{k+1})= \mu^k(x_1,\ldots,x_{k+1})+\varepsilon \eta_k\omega(x_1,\ldots,x_{k+1})
\end{align*}
$\eta_k\omega:\underbrace{\K\times\cdots\times\K}_{k+1}\to\K^n$ is the $\K$-multilinear map
\begin{align*}
\eta_k\omega(x_1,\ldots,x_{k+1})=\sum_{i=1}^k\mu^{k-i}(\omega(\mu^{i-1}(x_1,\ldots,x_{i}),x_{i+1}),x_{i+2},\ldots,x_{k+1}).
\end{align*}
Since $\mu^k(x_1,\ldots,x_{k+1})=0$ we obtain that
\begin{align*}
Z_{N_{n,k}}^2(\mu,\mu)=\{\omega\in Z_{K_n}^2(\mu,\mu): \eta_k\omega=0\}.
\end{align*}
We shall denote $H_{k\textrm{-nil}}^2(\mu,\mu):=H_{N_{n,k}}^2(\mu,\mu)$. 

By using the notation from Section 2.1, we can rewrite it as follows:
 \begin{align*}
 H^2_{k\textrm{-nil}}(\mu,\mu)&=\frac{Z_{N_{n,k}}^2(\mu,\mu)}{B^2(\mu,\mu)}\\
 &=\frac{Ker(\delta)\bigcap Ker(\eta_k)}{Im(\delta^1)},
 \end{align*}
 with $\delta:\Lambda^2({\K^n}^*)\rightarrow\Lambda^3({\K^n}^*)$ as before: 
 $$
 \delta\omega(x,y,z):=\sc \mu(\omega(x,y),z)+\sc \omega(\mu(x,y),z),
 $$ 
 $\delta^1:\Lambda^1({\K^n}^*)\rightarrow\Lambda^2({\K^n}^*)$:
 $$\delta^1(f)(x,y)=\mu(f(x),y)+\mu(x,f(y))-f(\mu(x,y)),$$
 and $\eta_k$ given by: 
 $$\eta_k(\omega)=
 \sum_{j=0}^{k-1}\mu^{k-1-j}\circ\omega\circ\mu^{j}.
 $$

 Note that this is the $k$-nil cohomology introduced in \cite{BCC}.

\begin{corollary}\label{coro:k-rig}
Let $\K$ be an algebraically closed field of characteristic zero and let $\N_{n,k}=N_{n,k}(\K)$, 
with the structure of affine algebraic variety. If $H^2_{k\textit{-nil}}(\mu,\mu)=0$, then $\mu$ is rigid in $\N_{n,k}$. 	
\end{corollary}

\begin{remark}
If $\K=\R$ and $H^2(\mu,\mu)=0$, then $O(\mu)$ is open in $(L_n)_{\textrm{red}}$ hence $O(\mu)(\mathbb{R})$ 
is open in $L_n(\mathbb{R})$ if we view $L_n(\mathbb{R})$ as $\mathbb{R}$-analytic space.  
A similar observation holds for $S_{n,k}$ and $N_{n,k}$. In particular we recover \cite[Theorem 2.1]{BCC}
\end{remark}

\medskip

\begin{acknowledgements}
 The authors would like to thank an anonymous referee for his or her thorough reading of this paper.  
 His or her comments made us improve significantly its first version. 
 This paper is part of the Ph.D. thesis of Josefina Barrio\-nuevo, being carried out thanks to a Doctoral Fellowship from CONICET, Argentina.
\end{acknowledgements}

\medskip



\begin{thebibliography}{BBC}
	
\bibitem[A]{A} Alvarez M.A.,
  \emph{On rigid 2-step nilpotent Lie algebras},
  Algebra Colloq.\ 25, No.\ 2 (2018), 349-360.
\bibitem[BCC]{BCC} Brega O., Cagliero L.\ and Chaves-Ochoa A.,
 \emph{The Nash–Moser theorem of Hamilton and rigidity of finite 
   dimensional nilpotent Lie algebras}, 
   Journal of Pure and Applied Algebra 221, (2017), 2250-2265.
 
\bibitem[AAA]{AAA} Arancibia B.\ Alfaro, Alvarez M.\ A.\ and Anza Y.,    
  \emph{Degenerations of graph Lie algebras}, 
   Linear and Multilinear Algebra (2020), DOI: 10.1080/03081087.2020.1712317
 
\bibitem[AG]{AG} Ancochea Bermudez J.M.\ and Goze M.,
 \emph{The rank of a linear system of roots of a complex solvable rigid Lie algebra} (French),
 Commun.\ Algebra 20 No.\ 3, (1992), 875--887.  
 
\bibitem[BS]{BS} Burde D.\ and Steinhoff C.,
 \emph{Classification of Orbit Closures of 4-Dimensional Complex Lie Algebras},
 Journal of Algebra Volume 214, Issue 2, (1999), 729-739.
 
\bibitem[BT]{BT} Barrionuevo J.\ and Tirao P.,
 \emph{Rigid 2-step graph Lie algebras},
 arXiv 2206.10572.

\bibitem[C]{C} Carles R., 
  \emph{Sur la structure des alg\`ebres de Lie rigides}, Annales de l’institut Fourier, 34(1984), 65-82.
  
\bibitem[GA]{GA} Goze M.\ and Ancochea Bermudez, 
 \emph{On the varieties of nilpotent Lie algebras of dimension 7 and 8}, 
  J.\ of Pure and Applied Algebra 77 (1992), 131–140. 

  
\bibitem[GR1]{GR1} Goze M.\ and Remm E.,
  \emph{$k$-step nilpotent Lie algebras},
  Georgian Math.\ J.\ 22, No.\ 2 (2015), 219-234.
  
\bibitem[GR2]{GR2} Goze M.\ and Remm E.,
  \emph{Lie algebras with associative structures. Applications to the study of 2-step nilpotent Lie algebras},
  arXiv 1201.2674v3 (2013).
  
  
\bibitem[GH1]{GH1} Grunewald F.\ and O’Halloran J., 
 \emph{Varieties of nilpotent Lie algebras of dimension less than six}, J.\ Algebra 112 (1988), 31–325.
  
  
\bibitem[GH2]{GH2} Grunewald F.\ and O'Halloran J.,
  \emph{Deformations of Lie Algebras}, Journal of Algebra 162, (1993), 210-224.

  
\bibitem[GT1]{GT1} Granada-Herrera F.\ and Tirao P.,
  \emph{Filiform Lie algebras of dimension 8 as degenerations}, Journal of algebras and its applications 13, (2014).
  
\bibitem[GT2]{GT2} Granada-Herrera F.\ and Tirao P.,  
 \emph{The Grunewald-O'Halloran conjecture for nilpotent Lie algebras of rank $\ge 1$},
 Comm.\ Alg.\ vol.\ 4 (2016), 2180-2192

\bibitem[H]{H} Hamilton R.S., 
 \emph{The inverse function theorem of Nash and Moser}, 
 Bull.\ Amer.\ Math.\ Soc., 7 (1982), 65-222.
  
\bibitem[LL]{LL} Leger and Lucks, 
 \emph{Cohomology of nilradicals of Borel subalgebras}, 
 Trans.\ Am.\ Math.\ Soc., 195 (1974), 305-316.
 
 \bibitem[MM]{MM} Martin Markl,
  Deformation Theory of Algebras and Their Diagrams,
  CBMS 116, AMS 2012.
 
 
 \bibitem[M]{Milne}
 Milne J. S., Algebraic Groups: The Theory of Affine Group Schemes of Finite Type over a Field, Cambridge University Press, 2017.
 
 \bibitem[NR]{NR} Nijenhuis A. and Richardson R.W.,
 \emph{Deformations of Lie Algebra Structures}
 Journal of Mathematics and Mechanics 17 (1), (1967), 89-105.
 
  
\bibitem[R]{R} Richardson, R.W.,
 \emph{On the rigidity of semi-direct products of Lie algebras},
 Pacific J.\ Math., Volume 22, Number 2 (1967), 339-344.
 
  \bibitem[RE]{Remm}
 Remm E., \emph{Rigid Lie algebras and algebraicity},
 Rev. Roumaine Math Pures Appl.\ 65, (2020), 491-510.
  
\bibitem[S]{S} Seeley C., 
 \emph{Degenerations of 6–dimensional nilpotent Lie algebras over 
 $\C$}, 
 Comm.\ Alg.\ 18 (1990), 3493–3505. 
  
\bibitem[TV]{TV} Tirao P.\ and Vera S.,
 \emph{There are no rigid filiform Lie algebras of low dimension},
 Journal of Lie Theory  vol.\ 29 (2019), 391-412.
  
  
\bibitem[V]{V} Vergne M.,
  \emph{Cohomologie des alg\`ebres de Lie nilpotentes. Application \`a l'\`etude de la vari\`et\`e des alg\`ebres de Lie nilpotentes}, Bulletin de la S. M. F., 98 (1970), 81-116. 
  
\end{thebibliography}
\end{document}